\documentclass[1p]{elsarticle}

\usepackage{amssymb, amsmath}

\usepackage{eucal}
\usepackage[utf8]{inputenc}

\usepackage{float}
\restylefloat{figure}

\usepackage{bbm}
\usepackage{dsfont}

\usepackage{hyperref}

\journal{Stochastic Processes and their Applications}

\usepackage{color}

\let\epsilon=\varepsilon
\let\phi=\varphi

\newcommand{\E}{\mathds{E}}
\newcommand{\V}{\mathds{V}}
\newcommand{\N}{\mathds{N}}
\newcommand{\R}{\mathds{R}}

\def\P{{\mathds P}}
\newcommand{\ind}{{\mathds 1}}

\newcommand{\Nc}{\mathcal{N}}

\newcommand{\xb}{{\bf x}}

\newcommand{\yb}{{\bf y}}

\newcommand{\Nb}{{\bf N}}

\newcommand{\Ent}{{\rm Ent}}

\newcommand{\Cf}{\mathfrak{C}}

\newtheorem{thm}{Theorem}[section]
\newtheorem{lem}[thm]{Lemma}
\newtheorem{prop}[thm]{Proposition}
\newtheorem{cor}[thm]{Corollary}
\newdefinition{rem}[thm]{Remark}
\newproof{proof}{Proof}

\begin{document}

\begin{frontmatter}

\title{Concentration for Poisson functionals:\\ component counts in random geometric graphs \tnoteref{mytitlenote}}
\tnotetext[mytitlenote]{\copyright \ 2015. This manuscript version is made available under the CC-BY-NC-ND 4.0 license \href{http://creativecommons.org/licenses/by-nc-nd/4.0/}{http://creativecommons.org/licenses/by-nc-nd/4.0/}.}

\author{Sascha Bachmann\corref{mycorrespondingauthor}}
\address{Institute for Mathematics, Osnabr\"uck University, 49069 Osnabr\"uck, Germany}
%\address{Institut f\"ur Mathematik, Universit\"at Osnabr\"uck, Germany}
%\fntext[myfootnote]{The author is partially supported by the German Research Foundation DFG-GRK 1916.}


\cortext[mycorrespondingauthor]{Corresponding author}
\ead{sascha.bachmann@uni-osnabrueck.de}


\begin{abstract}
Upper bounds for the probabilities $\P(F\geq \E F + r)$ and $\P(F\leq \E F - r)$ are proved, where $F$ is a certain component count associated with a random geometric graph built over a Poisson point process on $\R^d$. The bounds for the upper tail decay exponentially, and the lower tail estimates even have a Gaussian decay.

For the proof of the concentration inequalities, recently developed methods based on logarithmic Sobolev inequalities are used and enhanced. A particular advantage of this approach is that the resulting inequalities even apply in settings where the underlying Poisson process has infinite intensity measure.
\bigskip

\end{abstract}


\end{frontmatter}

%\linenumbers

\section{Introduction}
Random geometric graphs have been studied extensively for some decades now. In the simplest version of these graphs, the vertices are given by a random set of points in $\R^d$ and two vertices are connected by an edge if their distance is less than a fixed positive real number. This model was introduced by E. N. Gilbert in \cite{G_1961}, and since then many authors contributed to various directions of research on random geometric graphs. For a historical overview on the topic we refer the reader to the book \cite{P_2003} by M. D. Penrose. Recent contributions are e.g. \cite{DF_2011, LP_2013_1, LP_2013_2, RST_2013}.
\smallskip

It is a well established fact that numerous real world phenomena can be modeled by means of a random geometric graph, like for example the spread of a disease or a fire (see e.g. \cite{BW_2015, GG_2015}). Also, as communication networks such as wireless and sensor networks have become increasingly important in recent years, random geometric graphs have gained a considerable attention -- since they provide natural models for these objects (see e.g. \cite{CGY_2009, H_2013, MP_2010}).
\smallskip

Further applications arise from cluster analysis, where one aims to divide a given set of objects into groups (or clusters) such that objects within the same group are similar to each other (see e.g. \cite{B_1996_1, B_1996_2} for further reading). If the objects are represented by points in $\R^d$, one way to perform this task is to built a geometric graph over the points and to take the connected components of the graph as the clusters. At this, a connected component of a graph $G$ with vertex set $V$ is an induced connected subgraph $H$ of $G$ with vertex set $V'\subseteq V$ such that for any $x\in V'$ and $y\in V\setminus V'$ there is no edge between $x$ and $y$. For the purpose of statistical inference, a probabilistic theory for the connected components of the graph is needed.
\medskip

Throughout the present work, the vertices of the considered random geometric graphs are given by a Poisson point process on $\R^d$. The class of random variables that is investigated in this paper includes a variety of quantities that are typically of interest in several of the applications described above. For example, one can consider the number of connected components of the graph with at most $k$ (or alternatively with exactly $k$) vertices. Further random variables that are covered by our analysis are obtained by counting the number of components that are isomorphic to a fixed connected graph $H$. Early work on the latter quantities was done by R. Hafner in \cite{H_1972} and further related results are presented in \cite{P_2003}.
\medskip

The main contribution of the present paper is to establish new exponential upper bounds for the probabilities $\P(F\geq \E F + r)$ and $\P(F\leq \E F- r)$, where $\E F$ denotes the expectation of a component count $F$ and $r>0$ is a real number. Inequalities of this type are usually called \emph{concentration inequalities}. In order to derive our estimates, we use and enhance a method that was recently developed by S. Bachmann and G. Peccati in \cite{BP_2015}. The latter paper provides several refinements of a method for proving tail estimates for Poisson functionals (also known as the entropy-method), which is based on (modified) logarithmic Sobolev inequalities, and which was particularly studied in the seminal work by Wu \cite{W_2000}, extending previous findings by An\'e, Bobkov and Ledoux \cite{AL_2000, BL_1998}. Combining Wu's modified logarithmic Sobolev inequality with the famous Mecke formula for Poisson processes, the authors of \cite{BP_2015} were able to adapt concentration techniques for product space functionals, which were particularly developed by Boucheron, Lugosi and Massart \cite{BLM_2003}, and also by Maurer \cite{M_2006}, to the setting of Poisson processes. This approach adds a lot of flexibility to the entropy-method, and a remarkable feature of the obtained techniques is that they allow to deal with functionals build over Poisson processes with infinite intensity measure.

First applications for these techniques are worked out in \cite{BP_2015} and also in \cite{BR_2015}, where concentration bounds for certain Poisson U-statistics with positive kernels are established. A crucial property that was exploited in the latter investigations is that adding a point to the Poisson process cannot decrease the value of the considered functionals. In principle, this monotonicity is not needed for the method suggested in \cite{BP_2015} to be applied. However, due to somewhat more complicated objects that need to be controlled when dealing with non-monotonic functionals, the method has only been successfully used for monotonic quantities so far. \pagebreak Clearly, the component counts that are studied in the present paper are not monotonic. So, a particularly interesting aspect of the presented work is that it provides a class of functionals for which the approach from \cite{BP_2015} can be used although monotonicity is not satisfied.

As mentioned above, the used techniques do not restrict to finite intensity measure processes, and a remarkable feature of the presented estimates is indeed that they even apply in certain settings where the intensity measure of the underlying Poisson process is infinite, meaning that the resulting graph has almost surely infinitely many vertices. The following statement gathers several concentration estimates for component counts that are representative of the general bounds deduced in the present paper. To the best of the author's knowledge, there are no comparable concentration inequalities in the literature so far.

\begin{thm} \label{thm:1}
Let $\eta$ be a Poisson point process on $\R^d$ with locally finite and non-atomic intensity measure $\mu$. Let $\rho>0$ and consider the random geometric graph $G_\rho(\eta)$ with vertices $\eta$ and an edge between distinct $x,y\in\eta$ whenever $\lVert x-y\rVert \leq \rho$. Moreover, denote the ball centered at $x$ with radius $\rho$ by $B(x,\rho)$ and assume
\begin{align*}
\sigma_\rho^\mu = \sup_{x\in\R^d}\mu(B(x,\rho)) < \infty.
\end{align*}

Let $F$ be one of the following:
\begin{enumerate}[(i)]
\item the number of components of $G_\rho(\eta)$ with at most $k$ vertices;
\item the number of components of $G_\rho(\eta)$ with exactly $k$ vertices;
\item the number of components of $G_\rho(\eta)$ that are isomorphic to some fixed connected graph $H$ on $k$ vertices.
\end{enumerate}
Assume that $F$ is almost surely finite. Then $F$ is integrable and for any $r\geq 0$,
\begin{align*}
\P(F\geq \E F + r) &\leq \exp\left(-\frac{r^2}{a(2\E F + r)}\right),\\
\P(F\leq \E F - r) &\leq \exp\left(-\frac{r^2}{2 \max(a,4c_d / 3) \E F}\right),
\end{align*}
where $c_d>0$ is a constant that only depends on $d$ and
\begin{align*}
a = k\left(c_d^2 \sigma_\rho^\mu+ 1\right).
\end{align*}

\end{thm}

As an application of the established concentration inequalities, we prove strong laws of large numbers for suitably rescaled versions of the component counts. These results complement some statements from \cite[Chapter 3]{P_2003}, where comparable strong laws are proved for the $H$-component counts in the case where the vertices of the graph are i.i.d. points in $\R^d$. To the best of the author's knowledge, there are no strong laws in the literature so far for the Poisson case, and the established results even apply in certain settings where the underlying Poisson process has infinite intensity measure. Moreover, the strong laws from \cite{P_2003} do not cover at all the case where the expected degree of a typical vertex tends to infinity, commonly referred to as the \emph{dense regime}. \pagebreak In this regime, the limit behavior of the component counts typically depends heavily on the actual form of the intensity measure of the Poisson process. However, for a certain class of intensity measures, our strong laws hold even in the dense regime.
\smallskip

A further rapidly developing direction of research that is closely related to the study of random geometric graphs is the field of random geometric simplicial complexes (see e.g. \cite{DFR_2014, K_2011,  KM_2013, YSA_2014}). In order to investigate these random topological objects, it frequently turns out that results and properties of the underlying random geometric graphs can be very useful. Our findings and methods might therefore as well be of interest for future research on random simplicial complexes.
\smallskip

The paper is organized as follows. In Section 2, the framework that is considered in the present work is described . This includes a detailed description of the random geometric graph model as well as the introduction of the class of component counts that is at the core of our investigation. In Section 3, the main results of the present work are presented. In particular, the concentration inequalities for component counts, but also results regarding the limit behavior of the expectation, some integrability criteria for component counts as well as strong laws of large numbers are stated in this section. The proofs of our results are detailed in Section 4.

\section{Framework} \label{s:framework}

Let $\Nb$ denote the space of locally finite point configurations in $\R^d$. The elements of $\Nb$ can be regarded either as locally finite subsets of $\R^d$ or as locally finite simple counting measures on $\R^d$. For $\xi\in \Nb$ we will therefore use set notations like $\xi\cap A$ but also measure notations like $\xi(A)$. The space $\Nb$ comes equipped with the usual $\sigma$-algebra $\Nc$ that is generated by the maps $\Nb \to \R\cup\{\infty\}, \xi \mapsto \xi(A)$ where $A$ ranges over all Borel subsets of $\R^d$.
\medskip

Throughout, we will consider a non-trivial Poisson point process $\eta$ on $\R^d$ with locally finite and non-atomic intensity measure $\mu$. In particular, one has that $\eta$ is a random element in $\Nb$ that satisfies: \textbf{(i)} For any disjoint Borel sets $A_1,\ldots,A_n \subseteq\R^d$ the random variables $\eta(A_1),\ldots,\eta(A_n)$ are independent; \textbf{(ii)} For any Borel set $A\subseteq\R^d$ the random variable $\eta(A)$ is Poisson distributed with parameter $\mu(A)=\E\eta(A)$. Here we follow the convention that a Poisson random variable with infinite mean takes almost surely the value $\infty$.
\medskip

The geometric graph model that will be considered in the present work was particularly investigated in \cite{BR_2015, H_1972, LP_2013_1, LP_2013_2} and slightly generalizes the classical model of random geometric graphs. The latter model has been investigated by many authors and is extensively described in Penrose's book \cite{P_2003}. Let $S\subset\R^d$ be a Borel set such that $S = -S$. For any $\xi\in\Nb$ we define the \emph{geometric graph} $G_S(\xi)$ to be the graph with vertex set $\xi$ and an edge between two distinct vertices $x,y\in\xi$ whenever $x-y \in S$. Now, the \emph{random geometric graph} associated with $S$ and $\eta$ is given by $G_S(\eta)$. Let $B(x,\rho)$ denote the closed Euclidean ball centered at $x\in\R^d$ with radius $\rho>0$. It will be assumed throughout that $B(0, \rho)\subseteq S \subseteq B(0,\theta \rho)$ for some $\rho>0$ and $\theta \geq 1$. \pagebreak If we take $\theta=1$, i.e. $S=B(0,\rho)$, we obtain the classical random geometric graph with respect to the Euclidean norm that is also often referred to as \emph{random disk graph}. Of course, also any other norm ball can be chosen for $S$.
\medskip

We continue by introducing the class of random variables associated with $G_S(\eta)$ which are studied in the present work. For this purpose we first recall some basic concepts from graph theory. Consider some graph $G=(V, E)$ with vertex set $V$ and edge set $E\subseteq \{\{x,y\}\subseteq V:x\neq y\}$. An \emph{induced subgraph} of $G$ is a graph $H=(V', E')$ that satisfies $V'\subseteq V$ and moreover $E' = \{\{x,y\}\in E:x,y\in V'\}$. A \emph{(connected) component} of $G$ is an induced subgraph $H=(V', E')$ of $G$ such that $H$ is connected and for any $x\in V'$ and $y\in V\setminus V'$ one has $\{x,y\}\notin E$.

Now, for any $\xi\in\Nb$ we denote by $\Cf_S(\xi)$ the set of those subsets $\xb\subseteq \xi$ such that $G_S(\xb)$ is a connected component of $G_S(\xi)$. Moreover, for $k\in\N$ let
\begin{align*}
\Cf_S^k(\xi) = \{\xb\in \Cf_S(\xi) : |\xb| = k\} \ \text{ and } \ \Cf_S^{\leq k}(\xi) = \{\xb\in \Cf_S(\xi) : |\xb| \leq k\}.
\end{align*}
For any set $A\in\Nc^{\leq k} := \{B\in\Nc:|\xb|\leq k \text{ for all } \xb\in B\}$ we define the functional
\begin{align} \label{def:compcount}
F^A_S:\Nb\to \R\cup\{\infty\}, \ \ F^A_S(\xi) = \sum_{\xb\in \Cf_S^{\leq k}(\xi)} \ind\{\xb \in A\}.
\end{align}
The objects of study in the present work are the random variables $F_S^A(\eta)$ and if there is no risk of ambiguity, we will identify $F_S^A(\eta)$ with the corresponding functional $F_S^A$. Note that we assume throughout that the sets $A$ and $S$ as well as the intensity measure $\mu$ are chosen in a way such that $F_S^A(\eta)$ is non-trivial, i.e. $\E F_S^A(\eta) > 0$. The quantity $F_S^A$ counts all components in the graph $G_S(\eta)$ consisting of at most $k$ vertices that satisfy an arbitrary additional condition, given by the set $A$. This class of random variables includes many objects that naturally arise when studying random geometric graphs. For example, the number of connected components in $G_S(\eta)$ with at most $k$ vertices is obtained by taking $A=\{\xb\in\Nb: |\xb|\leq k\}$, while the choice $A=\{\xb\in\Nb: |\xb|=k\}$ yields the number of components with exactly $k$ vertices. Moreover, for any connected graph $H$ on $k$ vertices, one can take $A=\{\xb\in\Nb : G_S(\xb)\cong H\}$, where $\cong$ denotes an isomorphism between graphs. The resulting random variable $F^A_S =: J_H(\eta)$ is the so-called \emph{$H$-component count} associated with $G_S(\eta)$ which counts the connected components that are isomorphic to the graph $H$. A variety of results regarding this class of random variables can be found in \cite{H_1972} and also in \cite{P_2003}.
\medskip

We conclude this section with a picture that illustrates how the considered graphs and their connected components might look like in the plane. The intensity measure of the simulated Poisson process is given by the density $m(x) = 100(\lVert x\rVert + 1)^{-2}$ and the set $S$ is chosen to be the Euclidean ball with radius $\rho = 0.95$. The picture shows a realization of the resulting disk graph in the window $[-30,30]^2$ around the origin. Moreover, connected components with $3$ vertices are colored black.
\medskip

\begin{figure}[H]
\centering
\includegraphics[scale=0.73]{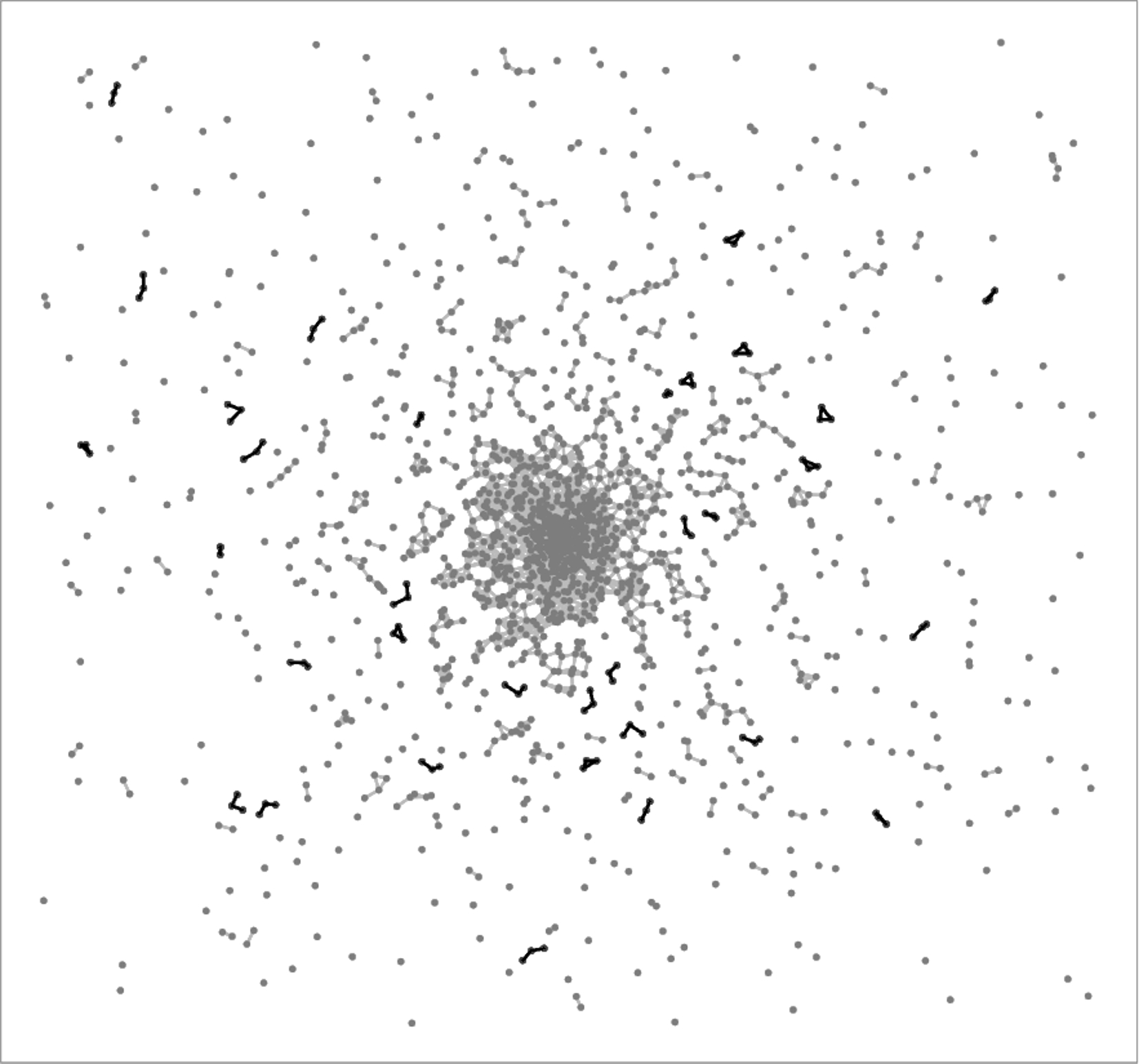}
\caption{Components with $3$ vertices of a random disk graph}
\end{figure}

\section{Main results}
\subsection{Concentration inequalities} \label{sec:CI}
The concentration inequalities for the random variables $F_S^A$ are obtained using methods for proving tail bounds for Poisson functionals that were recently developed in \cite{BP_2015}. We start our presentation of the main inequalities by citing from that article what is relevant for the upcoming discussion. To do so, we need to introduce the \emph{difference (or add-one cost) operator} $D$ which is defined for any measurable functional $F:\Nb\to\R$ by
\begin{align*}
D_xF(\xi) = F(\xi+\delta_x) - F(\xi), \ \ \text{for } (x,\xi)\in \R^d \times \Nb.
\end{align*}
Note also that for $z\in\R$ we will write $z_-^2 = \ind\{z<0\}z^2$ and $z_+^2 = \ind\{z>0\}z^2$. The following theorem is taken from \cite[Corollary 3.6 and Theorem 3.7]{BP_2015}.

\begin{thm} \label{thm:tool}
Let $F:\Nb\to\R$ be a measurable non-negative functional and consider the random variable $F(\eta)=F$. Assume that almost surely
\begin{align} \label{cond}
\int_{_{\R^d}} (D_xF(\eta))_-^2 d\mu(x) + \sum_{x\in\eta} (D_xF(\eta-\delta_x))_+^2 \leq a F(\eta)
\end{align}
for some constant $a>0$. Then $F$ is integrable and for any $r\geq 0$,
\begin{align*}
\P(F\geq \E F + r)&\leq \exp\left(-\frac{r^2}{a(2\E F + r)}\right).
\end{align*}
\end{thm}

The lower tail inequalities for the functionals $F_S^A$ will follow from the next result which is a generalization of \cite[Theorem 3.10]{BP_2015} to the case where the difference operator of the considered functional is allowed to have arbitrary sign. Note that this result holds even in the very general framework of the paper \cite{BP_2015} where Poisson point processes on arbitrary $\sigma$-finite measure spaces are considered.

\begin{thm} \label{thm:lowertail}
Let $F:\Nb\to\R$ be a measurable non-negative functional such that the random variable $F(\eta)=F$ is integrable. Assume that $F$ satisfies condition \eqref{cond} for some $a> 0$ and that moreover
\begin{align*}
|D_z F(\xi)| \leq 1 \ \ \text{for any} \ \ (z,\xi)\in\R^d\times\Nb.
\end{align*}
Then for all $r\geq 0$,
\begin{align*}
\P(F\leq \E F - r) \leq \exp\left(-\frac{r^2}{2 \max(a,4/3) \E F}\right).
\end{align*} 
\end{thm}

\begin{rem}
Note that in the above results, the functional $F$ is only allowed to take real values, while the functionals $F_S^A$ may as well take the value $\infty$. However, if the random variable $F_S^A(\eta)$ is almost surely finite, one can consider the functional $\mathfrak{f}_S^A:\Nb\to\R$ defined by $\mathfrak{f}_S^A(\xi) = F_S^A(\xi)$ if $F_S^A(\xi)<\infty$ and $\mathfrak{f}_S^A(\xi) = 0$ else. Then almost surely $F_S^A(\eta) = \mathfrak{f}_S^A(\eta)$ and via this modification the above results also apply to the functionals $F_S^A$.
\end{rem}

It will turn out that the difference operator of the functionals $F_S^A$ is bounded and that under quite general assumptions, these functionals satisfy the condition \eqref{cond} for a suitable constant $a>0$. We continue by introducing further quantities, depending on the parameters of the graph model, that will contribute to the constant $a$.

First, denote by $c_S$ the largest integer $k$ with the property that there exist $k$ points $x_1,\ldots,x_k \in S\subset \R^d$ such that $x_i - x_j \notin S$ whenever $i\neq j$. Then the assumption $B(0, \rho)\subseteq S \subseteq B(0,\theta \rho)$ ensures that $c_S < \infty$. Indeed, since $S$ is bounded, the maximal number of pairwise disjoint balls with radius $\rho/2$ and center in $S$ is finite. One also has that $\lVert x-y\rVert \leq \rho$ implies $x-y\in S$ for any $x,y\in\R^d$, where $\lVert \cdot\rVert$ denotes the Euclidean norm. Hence, we conclude that $c_S<\infty$. For the classical model of random geometric graphs where $S=B(0,\rho)$, the number $c_S$ coincides with the maximal number of points that can be placed in the unit ball in $\R^d$ such that any two of the points have distance larger than $1$. Then clearly $c_S$ depends only on the dimension $d$ of the surrounding space $\R^d$ and in the plane $\R^2$ one has for example $c_S=5$. 

We define a further constant $\sigma_S^\mu$ by
\begin{align*}
\sigma_S^\mu = \sup_{x\in\R^d}\mu(S+x).
\end{align*}
Of course, this quantity is not necessarily finite in such a general framework. However, the assumption $\sigma_S^\mu < \infty$, that will be in order throughout, still allows for a wide class of intensity measures $\mu$ that includes all finite intensity measures, but also even a homogeneous Poisson process verifies this condition. One crucial ancillary result for establishing concentration estimates for $F_S^A$ is the following.

\begin{thm}\label{thm:main}
Let $k \in\N$ and $A\in\Nc^{\leq k}$. Consider the random variable $F_S^A$ and assume that the intensity measure $\mu$ of the Poisson process $\eta$ ensures that almost surely $F_S^A<\infty$. Then $|D F_S^A| \leq c_S$ and almost surely
\begin{align*}
\int_{_{\R^d}} (D_xF_S^A(\eta))_-^2 d\mu(x) + \sum_{x\in\eta} (D_xF_S^A(\eta-\delta_x))_+^2 \leq a F_S^A,
\end{align*}
where
\begin{align*}
a = k\left(c_S^2 \sigma_S^\mu+ 1\right).
\end{align*}
\end{thm}

The above theorem together with Theorem \ref{thm:tool} and Theorem \ref{thm:lowertail} now immediately yields the following concentration bounds. Note that, in order to obtain the inequality for the lower tail, Theorem \ref{thm:lowertail} is applied to the functional $\tfrac {1}{c_S} F_S^A$.

\begin{cor}\label{cor:conc}
Let $k\in\N$ and $A\in\Nc^{\leq k}$. Consider the random variable $F_S^A$ and assume that the intensity measure $\mu$ of the Poisson process $\eta$ ensures that $\sigma_S^\mu<\infty$ as well as almost surely $F_S^A<\infty$. Then $F_S^A$ is integrable and for any $r\geq 0$,
\begin{align*}
\P(F_S^A\geq \E F_S^A + r) &\leq \exp\left(-\frac{r^2}{a(2\E F_S^A + r)}\right),\\
\P(F_S^A\leq \E F_S^A - r) &\leq \exp\left(-\frac{r^2}{2 \max(a,4c_S / 3) \E F_S^A}\right),
\end{align*}
where
\begin{align*}
a = k\left(c_S^2 \sigma_S^\mu+ 1\right).
\end{align*}
\end{cor}
The above result immediately implies Theorem \ref{thm:1}.

\subsection{Optimality}
We shall now briefly discuss optimality of the concentration bounds displayed in the above Corollary \ref{cor:conc}. The presented tail bounds are of the form $\exp(-I(r))$ where $I(r)$ is a function such that $\lim_{r\to\infty} I(r)/r^\alpha \in (0,\infty)$ for some $\alpha>0$. While our lower tail estimate has a fast Gaussian decay (meaning that $\alpha = 2$), the upper tail, however, only displays the exponent $\alpha = 1$, and it is natural to ask whether this can be improved.

It seems that the optimal exponent $\alpha$ for the upper tail actually depends on the concrete form of the intensity measure $\mu$ of the underlying Poisson process $\eta$. To give some (more or less vague) evidence for this phenomenon, we can consider first the situation where the intensity measure $\mu$ has bounded support. Then it is easy to see that the component count $F_S^A$ is almost surely bounded from above by some constant $C>0$, so the exponent of $\alpha=1$ that appears in our inequality for the upper tail is certainly not optimal in this case (our bound can then also be used to derive the estimate $\P(F_S^A\geq \E F_S^A + r) \leq \exp(-(a(2 \E F_S^A + C))^{-1}r^2)$, which has a Gaussian decay). On the other hand, if the support of the intensity measure $\mu$ is unbounded, the random variable $F_S^A$ may clearly take arbitrarily large values, meaning that the upper tail is thicker in this situation. Considering the previous observations, it seems likely that the optimal exponent for the upper tail bound depends on $\mu$, where a smaller exponent is to be expected when the mass of the intensity measure is widely spread out on the whole space $\R^d$. A closer investigation of this phenomenon would be an interesting direction for future research.

One approach to judge the quality of the constants appearing in the tail estimates is to compare them with the variance of $F_S^A$. More precisely, via computations similar to those in \cite[Sections 2.3 and 2.4]{blm_book}, one observes that if a random variable $Z$ satisfies for any $r>0$,
\begin{align*}
\P(Z\geq \E Z+r) &\leq \exp\left(-\frac{r^2}{v_1+wr}\right),\\
\P(Z\leq \E Z-r) &\leq \exp\left(-\frac{r^2}{2v_2}\right),
\end{align*}
then $\V Z \leq 2v_2 + 4v_1 + 8w^2$. In settings where the parameters of the model are varied, as it is done in Sections \ref{sec:ABE} and \ref{sec:SLLN} of the present paper, a natural question is now whether the asymptotic behavior of the expressions $2v_2$ and $4v_1+8w^2$  is of the same order as the variance. For the $H$-component counts $J_H$, to the best of the author's knowledge, the only result in the literature containing variance asymptotics is \cite[Proposition 3.8]{P_2003}, which deals with the thermodynamic regime, so we shall restrict our considerations to this case. According to the latter result together with \cite[Proposition 3.3]{P_2003}, one has $\lim_{t\to\infty}\V J_t/t\in(0,\infty)$ and also $\lim_{t\to\infty}\E J_t/t\in(0,\infty)$, where $(J_t = (J_H)_t)_{t\in\N}$ is a sequence of $H$-component counts associated with some appropriate sequence of Poisson processes $(\eta_t)_{t\in\N}$ and radii $(\rho_t)_{t\in\N}$ satisfying $\lim_{t\to\infty}t\rho_t^d \in(0,\infty)$ (see Section \ref{sec:ABE} for further details concerning the notation). Moreover, the constant $a = a_t$, which appears in our tail estimates (and which now of course depends on $t$), is bounded from above and from below in this situation (see the proof of Theorem \ref{thm:stronglaws} (i)). We therefore conclude that both $2v_2$ and $4v_1+8w^2$ are of the same order as the expectation, which in turn is of the same order as the variance.

\subsection{Integrability} \label{sec:Int}
Since the framework of the present paper is not restricted to finite intensity measure Poisson processes, the random variable $F_S^A$ is not necessarily integrable. Note that, according to Corollary \ref{cor:conc}, the assumption $c_S^\mu<\infty$ ensures that integrability of $F_S^A$ is equivalent to almost sure finiteness. In the following, we will characterize integrability (and thus almost sure finiteness) of $F_S^A$. To do so, we first mention that the random variable $F_S^A$ can be written as
\begin{align} \label{decompose}
F_S^A = \sum_{i = 1}^k F_S^{A_i}, \text{ where } A_i = \{\xb \in A, |\xb|=i\}.
\end{align}
Hence, the behavior of $\E F_S^A$ follows from the behavior of the expectations $\E F_S^{A_i}$. We will therefore consider without loss of generality only the case $A\in \Nc^k$ where
\begin{align*}
\Nc^k = \{B\in\Nc : |\xb|=k \text{ for } \xb\in B\},
\end{align*}
meaning that
\begin{align*}
F_S^A = \sum_{\xb\in\Cf^k_S(\eta)} \ind\{\xb\in A\}.
\end{align*}
The upcoming statement characterizes integrability of $F^A_S$ in terms of the integral
\begin{align} \label{expU}
\int_{(\R^d)^k} \ind\{\xb\in A, G_S(\xb) \ \text{is connected}\} d\mu^k(\xb),
\end{align}
where one should notice that here we use the symbol $\xb$ to denote a $k$-tuple in $(\R^d)^k$ instead of a $k$-element subset of $\R^d$. We prefer to not use different symbols for subsets and $k$-tuples since both play a very similar role in the context of the present paper. Moreover, we write $\xb\in A$ for some $\xb=(x_1,\ldots,x_k)\in\R^d$ to indicate that the corresponding subset $\{x_1,\ldots,x_k\}\subset \R^d$ is contained in $A$, and we write $G_S(\xb)$ to denote the graph $G_S(\{x_1,\ldots,x_k\})$.

By means of the Slivnyak-Mecke formula (see e.g. \cite[Corollary 3.2.3]{SW_2008}), the integral in \eqref{expU} coincides with the expectation of the random variable
\begin{align} \label{def:U}
U_S^A = \sum_{\xb\in\eta_{\neq}^k} \ind\{\xb\in A, G_S(\xb) \ \text{is connected}\},
\end{align}
where $\eta_ {\neq}^k$ denotes the set of $k$-tuples $(x_1,\ldots,x_k)\in\eta^k$ such that $x_i\neq x_j$ whenever $i\neq j$.

\begin{prop} \label{prop:integrable}
Let $k\in\N$ and $A\in \Nc^{k}$. Consider the random variable $F_S^A$ and assume that the intensity measure $\mu$ of the Poisson process $\eta$ ensures that $\sigma_S^\mu<\infty$. Then $\E F_S^A<\infty$ if and only if $\E U_S^A<\infty$. Moreover, $\E F_S^A>0$ if and only if $\E U_S^A>0$.
\end{prop}

In the case where the intensity measure $\mu$ is absolutely continuous with respect to the Lebesgue measure, we can use the above result to derive the following sufficient condition for integrability of $F_S^A$.

\begin{prop} \label{prop:integrableSuff}
Assume that the intensity measure $\mu$ of the Poisson process $\eta$ has a density $m$ with respect to the Lebesgue measure and that $\sigma_S^\mu<\infty$. Let $k\in\N$ and assume that
\begin{align} \label{cond:integrability}
\int_{\R^d} m(x)^k dx < \infty.
\end{align}
Then the random variable $F_S^A$ is integrable for any $A\in\Nc^k$.
\end{prop}

\subsection{Asymptotic behavior of the expectation} \label{sec:ABE}
The estimates in Corollary \ref{cor:conc} depend on the expectation of the random variable $F_S^A$, hence it is important for applications to know the asymptotic behavior of this quantity when the parameters of the model are varied. The upcoming results address this issue.

In the following, let again $S\subset\R^d$ be a set as described in Section \ref{s:framework}, where we assume without loss of generality that $B(0,1)\subseteq S \subseteq B(0,\theta)$ for some $\theta\geq 1$. Let $\mu$ be a non-trivial, locally finite and non-atomic measure on $\R^d$ that has a bounded Lebesgue density $m$. In particular, it holds that $\sigma_S^\mu<\infty$. Now, let $(\eta_t)_{t\in\N}$ be a sequence of Poisson point processes on $\R^d$ such that each $\eta_t$ has intensity measure $t\mu$. Also, let $(\rho_t)_{t\in\N}$ be a sequence of positive real numbers such that $\lim_{t\to\infty} \rho_t = 0$. Then for each $t\in\N$, we consider the random geometric graph $G_{\rho_t S}(\eta_t)$ associated with the set $\rho_t S$ and the point process $\eta_t$. For $k\in\N$ and $A\in\Nc^{\leq k}$, we are now interested in the asymptotic behavior of the random variables $F_t^A := F_{\rho_t S}^{\rho_tA}$ that are defined according to \eqref{def:compcount}. It will be assumed in what comes that the set $A$ is \emph{translation invariant}, meaning that $\xb+x\in A$ for any $\xb\in A$ and $x\in\R^d$. As it was pointed out above, it is enough to consider the case where $A\in\Nc^k$.

One prominent setting that is covered by the above framework is obtained by taking $A = \{\xb\in \Nb: G_S(\xb) \cong H\}$ for some fixed connected graph $H$ on $k$ vertices. Then the resulting random variables $F_t^A$ are exactly the $H$-component counts associated with the random geometric graphs $G_{\rho_t S}(\eta_t)$.

In order to apply the dominated convergence theorem in some proofs of the upcoming theorems, we need to make further assumptions on the density $m$. Therefore, we will often assume that $m$ is almost everywhere continuous and that there exist $\alpha,\gamma>0$ such that $m(x)\leq \alpha(\lVert x\rVert +1)^{-\gamma}$ for all $x\in\R^d$. In particular, this ensures that the density $m$ is bounded and if we assume in addition that $\gamma k> d$, then the condition in \eqref{cond:integrability} is verified which particularly guarantees that $F_t^A$ is integrable for all $t\in\N$.

\begin{rem}
The upcoming results for the sparse and thermodynamic regimes are obtained via a straightforward adaptation of the corresponding proofs and results presented in \cite[Chapter 3]{P_2003} to the more general framework of the present paper.
\end{rem}

\subsubsection{Sparse regime}

The behavior of quantities associated with random geometric graphs naturally depends heavily on how fast the sequence $(\rho_t)_{t\in\N}$ decays. Such a dependence can also be observed for the random variables $F_t^A$. We begin our investigation with the so-called \emph{sparse regime} where $t\rho_t^d \to 0$ as $t\to\infty$. In this regime, the asymptotic behavior of $\E F_t^A$ can be related to the random variables $U^A_t = U_{\rho S}^{\rho_t A}$ which are defined as in \eqref{def:U}. A similar phenomenon is described in \cite[Chapter 3]{P_2003}, and in particular in \cite[Proposition 3.2]{P_2003}, for the $H$-component counts built over i.i.d. points. Note that in the next result, we stick to the convention that $0/0=1$.

\begin{prop} \label{prop:UFsparse}
Let $k\in\N$ and $A\in \Nc^k$, where $A$ is translation invariant. Assume that the measure $\mu$ has a bounded Lebesgue density $m$ such that the random variables $F_t^A$ and $U_t^A$ are all integrable. Then, provided that $t\rho_t^d \to 0$ as $t\to\infty$, one has
\begin{align*}
\lim_{t\to\infty} \frac{k!\E F^A_t}{\E U^A_t} = 1.
\end{align*}
\end{prop}

So, in the sparse regime, the asymptotic behavior of $\E F_t^A$ follows from the asymptotics of the quantities $\E U_t^A$ and the latter sequence is usually easier to analyze. For instance, let $H$ be a connected graph on $k$ vertices and consider the $H$-component counts $(J_H)_t = J_t$ associated with $G_{\rho_t S}(\eta_t)$. Then the corresponding random variables $U_t^A$ are (up to rescaling by $k!$) the induced subgraph counts associated with the graphs $G_{\rho_t S}(\eta_t)$. The limit behavior of the expectation of these subgraph counts is well studied. The statement in \cite[Proposition 3.1]{P_2003} covers for example the case where the measure $\mu$ is finite and has a bounded and almost everywhere continuous Lebesgue density. Generalizing the approach from \cite{P_2003} yields the following result.

\pagebreak

\begin{thm}\label{thm:expsparse}
Let $k \in\N$ and $A\in\Nc^{k}$, where $A$ is translation invariant and consider the random variables $(U_t^A)_{t\in\N}$. Assume that the Lebesgue density $m$ of the measure $\mu$ is almost everywhere continuous and that there are $\alpha,\gamma>0$ where $\gamma k> d$ such that $m(x) \leq \alpha(\lVert x\rVert +1)^{-\gamma}$ for all $x\in\R^d$. Then all $U_t^A$ are integrable and
\begin{align}\label{eq:lim_uta}
\lim_{t\to\infty} \frac{\E U_t^A}{t^k \rho_t^{d(k-1)}} = k! \ \mathfrak{s}_S^A(m) < \infty,
\end{align}
where
\begin{align*}
\mathfrak{s}_S^A(m) := \frac{1}{k!} \int_{\R^d} m(x)^k dx \int_{(\R^d)^{k-1}} I(\xb) d\xb
\end{align*}
and
\begin{align*}
I(\xb)= \ind\{\{0,x_1,\ldots,x_{k-1}\}\in A, G_S(0,x_1,\ldots,x_{k-1}) \ \text{\rm is connected}\}.
\end{align*}
If moreover $t\rho_t^d \to 0$ as $t\to\infty$, then
\begin{align}\label{eq:lim_fta}
\lim_{t\to\infty} \frac{\E F_t^A}{t^k \rho_t^{d(k-1)}} = \mathfrak{s}_S^A(m).
\end{align}
\end{thm}

\subsubsection{Thermodynamic regime} \label{ssec:thermo}
We continue with the case where the sequence $t\rho_t^d$ converges to a positive constant, i.e. $t\rho_t^d \to {\rm const}\in(0,\infty)$ as $t\to\infty$. This is referred to as the \emph{thermodynamic regime}. In this regime, one can still analyze the asymptotic behavior of the expectations $\E F_t^A$ for a large class of intensity  measures. For the $H$-component counts, the statement \cite[Proposition 3.3]{P_2003} describes these asymptotics in the case where the intensity measure is finite and has a bounded and almost everywhere continuous Lebesgue density. The proof of the latter result can be adapted to cover also the more general class of random variables $F_t^A$. Note that in the following, the Lebesgue measure on $\R^d$ will be denoted by $\lambda$.

\begin{thm}\label{thm:expthermo}
Let $k \in\N$ and $A\in\Nc^{k}$, where $A$ is translation invariant and consider the random variables $(F_t^A)_{t\in\N}$. Assume that the Lebesgue density $m$ of the measure $\mu$ is almost everywhere continuous and that there are $\alpha,\gamma>0$ where $\gamma k> d$ such that $m(x) \leq \alpha(\lVert x\rVert +1)^{-\gamma}$ for all $x\in\R^d$. Then all $F_t^A$ are integrable and, provided that $t\rho_t^d \to c\in(0,\infty)$ as $t\to\infty$, one has
\begin{align*}
\lim_{t\to\infty} \frac{\E F_t^A}{t} = \mathfrak{t}_S^A(m)<\infty,
\end{align*}
where
\begin{align*}
\mathfrak{t}_S^A(m) := \frac{c^{k-1}}{k!} \int_{(\R^d)^k} I(\xb) m(x_1)^k e^{-c \lambda(S\cup(S+x_2)\cup\ldots\cup(S+x_k))m(x_1)} d\xb
\end{align*}
and
\begin{align*}
I(\xb)= \ind\{\{0,x_2,\ldots,x_k\}\in A, G_S(0,x_2,\ldots,x_k) \ \text{\rm is connected}\}.
\end{align*}
\end{thm}

\subsubsection{Dense regime}
The case where the sequence $t \rho_t^d$ tends to infinity as $t\to\infty$ is commonly referred to as the \emph{dense regime}. In contrast to the situation in the sparse and thermodynamic regimes, the asymptotic behavior of the expectations of $F_t^A$ in the dense regime depends heavily on the concrete form of the intensity measure $\mu$. We will therefore restrict our investigation to the case where the intensity measure is given by a Lebesgue density $m(x)=\alpha(\lVert x\rVert +1)^{-\gamma}$ for some $\alpha,\gamma>0$.

\begin{thm}\label{thm:expdense}
Let $k \in\N$ and $A\in\Nc^{k}$, where $A$ is translation invariant and consider the random variables $(F_t^A)_{t\in\N}$. Assume that the measure $\mu$ is given by a Lebesgue density
\begin{align*}
m(x) = \alpha(\lVert x\rVert +1)^{-\gamma},
\end{align*}
where $\alpha,\gamma > 0$ and $\gamma k>d$. Then the $F_t^A$ are integrable and, provided that $t\rho_t^d \to \infty$ as $t\to\infty$, one has
\begin{align*}
\lim_{t\to\infty} \frac{\E F_t^A}{t(t\rho_t^d)^{d/\gamma - 1}} = \mathfrak{d}_S^A(m) < \infty,
\end{align*}
where
\begin{align*}
\mathfrak{d}_S^A(m) := \frac{\alpha^k}{k!} \int_{(\R^d)^k} I(\xb) \lVert x_1\rVert^{-\gamma k} e^{-\alpha \lVert x_1\rVert^{-\gamma} \lambda \left(S \cup \bigcup_{i=2}^k \left(S + x_i\right)\right)} d\xb
\end{align*}
and
\begin{align*}
I(\xb) = \ind\{\{0,x_2,\ldots,x_k\}\in A, G_S(0,x_2,\ldots,x_k) \ \text{\rm is connected}\}.
\end{align*}
\end{thm}

\subsection{Strong laws of large numbers} \label{sec:SLLN}
Let the conventions of Section \ref{sec:ABE} prevail. The tail bounds from Corollary \ref{cor:conc} together with the asymptotic behavior of the expectation stated in the results from Section \ref{sec:ABE} yield strong laws for the component counts $F_t^A$. Statement (i) of the theorem below complements the results \cite[Theorem 3.15, Theorem 3.16 and Theorem 3.19]{P_2003} that provide similar strong laws for the $H$-component counts associated with random geometric graphs built over i.i.d. points in $\R^d$. In the latter results, the role of the intensity scaling factor $t$ that is used in the present framework is taken over by the number of i.i.d. points. Via this correspondence, the strong laws presented in statement (i) of the upcoming theorem, when specialized to the $H$-component counts, are Poisson-space analogues for the results mentioned above.
 
Penrose's result \cite[Theorem 3.15]{P_2003} provides a strong law for the thermodynamic regime, and \cite[Theorem 3.16]{P_2003} contains a strong law for the sparse regime where $\lim_{t\to\infty}t\rho_t^d = 0$ and $\lim_{t\to\infty} t^{2k-1} \rho_t^{d(2k-2)} / \log(t) = \infty$. The latter condition means that $t^k\rho_t^{d(k-1)}$ grows faster than $\sqrt{t\log(t)}$, which is more restrictive than the condition \eqref{condRegimeST} in the theorem below. The very sparse regime is considered in Penrose's result \cite[Theorem 3.19]{P_2003}, which requires that there is some $\tau>0$ such that for large enough $t$, one has $t^k\rho_t^{d(k-1)} > (\log(t))^{1+\tau}$ and $t \rho_t^d< t^{-\tau}$, and also requires that the sequence $(\rho_t)_{t\in\N}$ is \emph{regularly varying}, meaning that $\lim_{t\to\infty} \rho_{\lfloor ts \rfloor}/\rho_t\in(0,\infty)$ for all $s>0$.

It is worth noting that the result below is independent of the joint distribution of the random variables $F_t^A$, so the almost sure convergence is actually a complete convergence. In contrasts to this, while the results \cite[Theorem 3.15 and Theorem 3.16]{P_2003} also provide complete convergence, the result \cite[Theorem 3.19]{P_2003} only provides almost sure \linebreak convergence for the joint distribution that is obtained when successively adding i.i.d. points in $\R^d$.

\begin{thm} \label{thm:stronglaws}
Let $k \in\N$ and $A\in\Nc^{k}$, where $A$ is translation invariant. Consider the associated sequence of random variables $(F_t^A)_{t\in\N}$. Let the measure $\mu$ be given by a Lebesgue density $m$. Then the following statements hold:
\begin{enumerate}[(i)]
\item \emph{sparse and thermodynamic regime}. Assume that $m$ is almost everywhere continuous and that there are $\alpha,\gamma>0$ where $\gamma k> d$ such that $m(x) \leq \alpha(\lVert x\rVert +1)^{-\gamma}$ for all $x\in\R^d$. Assume in addition that $t\rho_t^d\to {\rm const} \in [0,\infty)$ as $t\to\infty$ and that
\begin{align} \label{condRegimeST}
\lim_{t\to\infty}\frac{t^k\rho_t^{d(k-1)}}{\log(t)}=\infty.
\end{align}
Then, if $\lim_{t\to\infty} t \rho_t^d = 0$, one has
\begin{align*}
\frac{F_t^A}{t^k \rho_t^{d(k-1)}} \overset{a.s.}{\longrightarrow} \mathfrak{s}_S^A(m) \ \text{ as } \ t\to\infty.
\end{align*}
Moreover, if $\lim_{t\to\infty} t \rho_t^d \in (0,\infty)$, one has
\begin{align*}
\frac{F_t^A}{t} \overset{a.s.}{\longrightarrow} \mathfrak{t}_S^A(m) \ \text{ as } \ t\to\infty.
\end{align*}
\item \emph{dense regime}. Let the density $m$ be given by $m(x) = \alpha(\lVert x\rVert + 1)^{-\gamma}$ for some $\alpha,\gamma>0$ and $\gamma k>d$. Assume that $t\rho_t^d\to\infty$ as $t\to\infty$ and that moreover 
\begin{align} \label{condRegimeD}
\lim_{t\to\infty}\frac{t(t\rho_t^d)^{d/\gamma-2}}{\log(t)}=\infty.
\end{align}
Then
\begin{align*}
\frac{F_t^A}{t(t\rho_t^d)^{d/\gamma-1}} \overset{a.s.}{\longrightarrow} \mathfrak{d}_S^A(m) \ \text{ as } \ t\to\infty.
\end{align*}
\end{enumerate}
\end{thm}

\section{Proofs}

\subsection{Proofs for the concentration inequalities}
The analytic lemma below is used in the upcoming proof of Theorem \ref{thm:lowertail}.

\begin{lem} \label{lem:analytic}
Let $\psi(z) = ze^z - e^z + 1$. Then for any $a>0$ and $z>0$, one has
\begin{align*}
\frac{a\psi(z)/z^ 2}{1+a\psi(z)/z} \leq \frac{\max(a,4/3)}{2}.
\end{align*}
\end{lem}

\begin{proof}
The desired inequality can be rearranged as
\begin{align*}
a \psi(z) (1-cz) \leq cz^2,
\end{align*}
where $c=\max(a,4/3)/2$. It will be established below that
\begin{align} \label{eq:psibound}
\psi(z)\left(1-\tfrac 23 z\right)\leq \tfrac 12 z^2.
\end{align}
In the case $a\geq 4/3$, one has $c=a/2$. Hence, using \eqref{eq:psibound} we obtain
\begin{align*}
a \psi(z) (1-cz) =  a \psi(z) \left(1-\tfrac a2 z\right)\leq a \psi(z) \left(1-\tfrac 23 z\right) \leq  \tfrac a2 z^2 = cz^2,
\end{align*}
so the result holds in this case. Now assume that $a< 4/3$. Then $c=2/3$ and using again \eqref{eq:psibound} yields
\begin{align*}
a \psi(z) (1-cz) =  a \psi(z) \left(1-\tfrac 23 z\right) \leq  \tfrac a2 z^2 \leq cz^2.
\end{align*}
It remains to prove \eqref{eq:psibound}. To do so, we first rearrange this inequality as
\begin{align*}
\tfrac 53 ze^z+1 \leq \tfrac 23 z^2 e^z + e^z + \tfrac 23 z + \tfrac 12 z^2.
\end{align*}
To prove the above, we compute
\begin{align*}
\tfrac 53 z e^z + 1 &= \tfrac 53 \sum_{n\geq 0} \frac{z^{n+1}}{n!} + 1 = \tfrac 23 \sum_{n\geq 0} \frac{z^{n+1}}{n!} + \sum_{n\geq 0} \frac{(n+1)z^{n+1}}{(n+1)!} + 1 \\ &= \tfrac 23 \sum_{n\geq 0} \frac{z^{n+1}}{n!} + \sum_{n\geq 1} n \frac{z^{n+1}}{(n+1)!} + \sum_{n\geq 0} \frac{z^{n+1}}{(n+1)!} + 1 \\ &= \tfrac 23 \sum_{n\geq 1} \frac{z^{n+1}}{n!} + \tfrac 23 z + \sum_{n\geq 2} n \frac{z^{n+1}}{(n+1)!} + \tfrac 12 z^2 + e^z \\ &= \tfrac 23 z^2 \left(1 + \sum_{n\geq 1} \frac{z^{n}}{(n+1)!} + \tfrac 32 \sum_{n\geq 1} (n+1) \frac{z^{n}}{(n+2)!} \right) + e^z + \tfrac 23 z + \tfrac 12 z^2.
\end{align*}
Now, the last expression in the above display can be upper bounded by
\begin{align*}
\tfrac 23 z^2 \left(1 + \sum_{n\geq 1} \frac{z^{n}}{n!}\right) + e^z + \tfrac 23 z + \tfrac 12 z^2 = \tfrac 23 z^2 e^z + e^z + \tfrac 23 z + \tfrac 12 z^2,
\end{align*}
where we used the obvious estimate
\begin{align*}
\frac{1}{(n+1)!} + \frac{3(n+1)}{2(n+2)!} \leq \frac{1}{n!}.
\end{align*}
\qed\end{proof}

The crucial ingredient in the upcoming proof of Theorem \ref{thm:lowertail} is the following logarithmic Sobolev inequality that is a special case of \cite[Proposition 3.1]{BP_2015}, which is in turn obtained by combining Wu's modified logarithmic Sobolev inequality (see \cite[Corollary 2.3]{W_2000}) with the Mecke formula for Poisson processes (see \cite[Satz 3.1]{M_1967}). The result uses the \emph{entropy} of an integrable random variable $Z>0$, defined by
\begin{align*}
\Ent(Z) = \E(Z\log(Z)) - \E(Z)\log(\E Z).
\end{align*}
Note also that for $z\in\R$ we will write $z_- = \ind\{z<0\} z$ and $z_+ = \ind\{z>0\} z$.

\begin{prop} \label{EntIneq}
Let $F:\Nb\to\R$ be a measurable functional and consider the random variable $F(\eta)=F$. Then for all $u \in\R$ satisfying $\E (e^{u F})<\infty$ we have
\[
\Ent(e^{u F}) \leq \E \left[ e^{u F} \left(\int_{\R^d} \psi(u D_x F(\eta)_-) \ d\mu(x) + \sum_{x\in\eta} \phi(-u D_x F(\eta-\delta_x)_+)\right)\right],
\]
where $\phi(z) = e^z - z - 1$ and $\psi(z) = ze^z - e^z + 1$.
\end{prop}

The proof of Theorem \ref{thm:lowertail} below is very similar to the proof of \cite[Theorem 3.10]{BP_2015} which in turn is an adaptation of the proof of \cite[Theorem 13]{M_2006} for Poisson functionals. For the sake of completeness, we carry out the modified argumentation.

\newproof{po_thm_lowertail}{Proof of Theorem \ref{thm:lowertail}}
\begin{po_thm_lowertail}
For brevity, we will only deal with the case where $F$ is boun\-ded here and remark that the result can be extended to the unbounded case in exactly the same way as in the proof of \cite[Theorem 3.10]{BP_2015}. 

By Proposition \ref{EntIneq} we have for any $u< 0$,
\begin{align*}
\Ent(e^{u F}) &\leq \E \left[ e^{u F} \left(\int_{\R^d} \psi(u D_x F(\eta)_-) \ d\mu(x) + \sum_{x\in\eta} \phi(-u D_x F(\eta-\delta_x)_+)\right)\right].
\end{align*}
Moreover, by assumption we have $|D_x F(\eta)| \leq 1$, thus $u D_x F(\eta)_- \leq -u$. Since $\psi(z)/z^2$ is increasing in $z$, it follows that
\begin{align*}
\int_{\R^d} \psi(u D_x F(\eta)_-) d\mu(x) &=u^2 \int_{\R^d} \frac{\psi(u D_x F(\eta)_-)}{u^2 D_x F(\eta)_-^2} D_x F(\eta)_-^2\ d\mu(x)\\
&\leq u^2 \int_{\R^d} \frac{\psi(-u) }{u^2} D_x F(\eta)_-^2\ d\mu(x).
\end{align*}
Similarly, one also has
\begin{align*}
\sum_{x\in\eta} \phi(-u D_x F(\eta)_+) \leq u^2 \sum_{x\in\eta} \frac{\phi(-u) }{u^2} D_x F(\eta-\delta_x)_+^2.
\end{align*}
Now, since $\phi(-u)\leq \psi(-u)$ for any $u<0$ and since by assumption
\begin{align*}
\int_{\R^d} D_x F(\eta)_-^2\ d\mu(x) + \sum_{x\in\eta} D_x F(\eta-\delta_x)_+^2 \leq aF, 
\end{align*}
it follows that $\Ent(e^{u F}) \leq \psi(-u) a \E(Fe^{u F})$. Dividing this inequality by $u^2\E(e^{u F})$ gives
\begin{align} \label{eq:h_ineq}
h'(u) = \frac{\Ent(e^{u F})}{u^2\E(e^{u F})} \leq \frac{\psi(-u) a \E(Fe^{u F})}{u^2\E(e^{u F})},
\end{align}
where $h(u) = u^{-1} \log \E(e^{uF})$. Let $\nu<0$. Integrating inequality \eqref{eq:h_ineq} from $\nu$ to $0$ and using that, since $\psi(-u)/u^2$ is decreasing in $u$, one has $\psi(-u)/u^2 \leq \psi(-\nu)/\nu^2$ for all $u\in[\nu,0)$, yields
\begin{align*}
\E F - \frac{1}{\nu} \log\E(e^{\nu F}) \leq -\frac{\psi(-\nu)}{\nu^ 2} a \log\E(e^ {\nu F}).
\end{align*}
Since $1- a\psi(-\nu)/\nu$ is positive, we can rearrange the above inequality as
\begin{align*}
\log\E[\exp(\nu(F-\E F))] \leq \nu^2 \frac{a\psi(-\nu)/\nu^2}{1- a\psi(-\nu)/\nu} \E F.
\end{align*}
Moreover, by Lemma \ref{lem:analytic} we have
\begin{align*}
\frac{a\psi(-\nu)/\nu^2}{1- a\psi(-\nu)/\nu} \leq \frac{\max(a,4/3)}{2}.
\end{align*}
The last two displays together with Markov's inequality yield
\begin{align*}
\P(F\leq \E F - r) \leq \E[\exp(\nu(F-\E F))]e^{\nu r} \leq \exp\left(\nu^2 \frac{\max(a,4/3)}{2} \E F + \nu r\right).
\end{align*}
The result is now obtained by an easy optimization in $\nu$.
\qed\end{po_thm_lowertail}

To get prepared for the upcoming proof of Theorem \ref{thm:main}, we first gather some observations in the following lemma.

\begin{lem} \label{lem:observations}
Let $k \in\N$ and $A\in\Nc^{\leq k}$. Consider the functional $F_S^A$ and let $\xi\in\Nb$ be such that $F_S^A(\xi)<\infty$. Then the following statements hold:
\begin{enumerate}[(i)]
\item For any $x\in\R^d$ one has $|D_xF_S^A(\xi)_-| \leq c_S$.
\item If $x\in\R^d$ is such that $x\notin \cup_{y\in\yb} (S+y)$ for all $\yb\in\Cf_S^{\leq k}(\xi)\cap A$, then it follows that $D_xF_S^A(\xi)\geq 0$.
\item For any $x\in\xi$ one has $D_xF_S^A(\xi-\delta_x)_+ \leq 1$. Moreover, if equality holds in the latter inequality, then $x\in\xb$ for some $\xb\in \Cf_S^{\leq k}(\xi)\cap A$.
\end{enumerate}
\end{lem}

\begin{proof}

[Proof of (i)] Let $K$ be the set of elements $\yb\in \Cf_S^{\leq k}(\xi)$ such that there exists a vertex $y\in\yb$ with $y-x\in S$. Adding the point $x$ to the set $\xi$ does not affect all those components of the graph $\mathfrak{G}_S(\xi)$ that correspond to the sets $\yb\in \Cf_S^{\leq k}(\xi) \setminus K$. It follows that $|D^-_xF_S^A(\xi)| \leq |K|$. By definition of $K$, for every $\yb\in K$ we can choose a point $v(\yb)\in\yb$ such that $v(\yb) - x\in S$. Since any distinct $\yb,\yb'\in K$ correspond to two different components of $\mathfrak{G}_S(\xi)$, the vertices $v(\yb), v(\yb')$ are not connected by an edge. Thus, for any distinct $\yb,\yb'\in K$, we have $(v(\yb)-x)-(v(\yb')-x)= v(\yb)-v(\yb')\notin S$. It follows from the definition of $c_S$ that $|K|\leq c_S$.

[Proof of (ii)] The assumption on $x$ ensures that adding the point $x$ to the set $\xi$ does not affect all those components of $G_S(\xi)$ that correspond to the vertex sets in $\Cf_S^{\leq k}(\xi) \cap A$. Hence $\Cf_S^{\leq k}(\xi) \cap A \subseteq \Cf_S^{\leq k}(\xi+\delta_x) \cap A$ and this implies $D_x F_S^A(\xi)\geq 0$.

[Proof of (iii)] Of course, we have $D_x F_S^A(\xi-\delta_x)\leq 1$ since $x$ belongs to exactly one component of $G_S(\xi)$. Moreover, if the component of $G_S(\xi)$ that contains the vertex $x$ does not correspond to a vertex set in $\Cf_S^{\leq k}(\xi)\cap A$, then $\Cf_S^{\leq k}(\xi)\cap A \subseteq \Cf_S^{\leq k}(\xi-\delta_x)\cap A$ and thus $D_x F_S^A(\xi-\delta_x)\leq 0$. This implies the additional statement.
\qed\end{proof}

We are now equipped for the proof of Theorem \ref{thm:main}.

\newproof{po_thm_main}{Proof of Theorem \ref{thm:main}}
\begin{po_thm_main}
The bound on the difference operator $DF_S^A$ follows immediately from Lemma \ref{lem:observations} (i) and (iii).

Using Lemma \ref{lem:observations} (i) we obtain
\begin{align*}
\int_{\R^d} (D_xF_S^A(\eta))^2_- d\mu(x) \leq c_S^2 \mu(\{x\in\R^d : D_xF_S^A(\eta)<0\}).
\end{align*}
According to Lemma \ref{lem:observations} (ii),
\begin{align*}
\mu(\{x\in\R^d : D_xF_S^A(\eta)<0\}) &\leq \mu\left(\cup_{\yb\in\Cf_S^{\leq k}(\eta)\cap A} \cup_{y\in\yb}(S+y)\right)\\ &\leq k \sigma^\mu_S |\Cf_S^{\leq k}(\eta)\cap A|\\ &= k \sigma^\mu_S F_S^A.
\end{align*}
Moreover, by virtue of Lemma \ref{lem:observations} (iii), one has
\begin{align*}
\sum_{x\in\eta} (D_xF_S^A(\eta-\delta_x))_+^2 \leq \left|\cup_{\xb\in\Cf_S^{\leq k}(\eta)\cap A} \xb\right| \leq k |\Cf_S^{\leq k}(\eta)\cap A| = k F_S^A.
\end{align*}
Combining the above estimates yields the result.
\qed\end{po_thm_main}

\subsection{Proofs for the integrability criteria}

We start with a preliminary lemma that is needed in several of the upcoming proofs.

\begin{lem} \label{lem:prelim}
Let $k\in\N$ and $A\in \Nc^{k}$. Then the expectation of the random variable $F_S^A$ is given by
\begin{align*}
\E F_S^A = \frac{1}{k!} \int_{(\R^d)^k} \ind\{\xb\in A, G_S(\xb) \text{ \rm is connected}\} e^{-\mu\left(\cup_{i=1}^k(S+x_i)\right)} d\mu^ k(\xb).
\end{align*}
\end{lem}

\begin{proof}
First note that for any $\xi\in\Nb$ and $\xb=(x_1, \ldots,x_k)\in(\R^d\setminus \xi)_{\neq}^k$, the property that $G_S(\xb)$ is a connected component of $G_S(\xi \cup \{x_1,\ldots,x_k\})$ is equivalent to
\begin{align*}
G_S(\xb) \text{ is connected and } \xi \cap \left(\cup_{i=1}^k (S+x_i)\right) = \emptyset.
\end{align*}
Thus, using the Slivnyak-Mecke formula (see e.g. \cite[Corollary 3.2.3]{SW_2008}) together with the Fubini theorem, we obtain
\begin{align*}
\E F_S^A &= \E \sum_{\xb\in\Cf^k_S(\eta)} \ind\{\xb\in A\}\\ &= \frac{1}{k!}\E \sum_{\xb\in\eta_{\neq}^k} \ind\{\xb\in A, G_S(\xb) \text{ is connected component of } G_S(\eta)\}\\
&= \frac{1}{k!} \int_{(\R^d)^k} \ind\{\xb\in A, G_S(\xb) \text{ is connected}\} \P\left(\eta \cap \left(\cup_{i=1}^k (S+x_i)\right) = \emptyset\right) d\mu^k(\xb).
\end{align*}
Moreover, since $\eta$ is a Poisson process, the random variable $\left|\eta \cap \left(\cup_{i=1}^k (S+x_i)\right)\right|$ is Poisson distributed with mean $\mu\left(\cup_{i=1}^k (S+x_i)\right)$, hence
\begin{align*}
\P\left(\eta \cap \left(\cup_{i=1}^k (S+x_i)\right) = \emptyset\right) = e^{-\mu\left(\cup_{i=1}^k (S+x_i)\right)}.
\end{align*}
\qed\end{proof}

\newproof{po_prop_integrable}{Proof of Proposition \ref{prop:integrable}}
\begin{po_prop_integrable}
For any element $\xb = (x_1,\ldots,x_k) \in(\R^d)^k$ one has that $\mu\left(\cup_{i=1}^k (S+x_i)\right) \leq k \sigma_S^\mu < \infty$ and thus
\begin{align*}
1\geq e^{-\mu\left(\cup_{i=1}^k (S+x_i)\right)} \geq e^{-k\sigma_S^\mu}>0.
\end{align*}
Using this, the result follows from Lemma \ref{lem:prelim} together with the fact that the expectation of $U_S^A$ is given by \eqref{expU}.
\qed\end{po_prop_integrable}

\newproof{po_prop_integrableSuff}{Proof of Proposition \ref{prop:integrableSuff}}
\begin{po_prop_integrableSuff}
We deduce from Proposition \ref{prop:integrable} that integrability of $F_S^A$ is implied by finiteness of the integral
\begin{align*}
\int_{(\R^d)^k} \ind\{G_S(\xb) \text{ \rm is connected}\} m^{\otimes k}(\xb)d\xb,
\end{align*}
where
\begin{align*}
m^{\otimes k}(\xb) = m^{\otimes k}(x_1,\ldots,x_k) = \prod_{i=1}^k m(x_i).
\end{align*}
Denote by $\lambda$ the Lebesgue measure on $\R^d$. To see that the above integral is finite, we compute
\begin{align}\label{mkcomputation}
&\int_{(\R^d)^k} \ind\{G_S(\xb) \text{ is connected}\} m^{\otimes k}(\xb)d\xb \\ \nonumber &\leq \int_{(\R^d)^k} \sum_{i=1}^k m(x_i)^k \ \ind\{G_S(\xb) \text{ \rm is connected}\} d\xb \\ \nonumber  &= \sum_{i=1}^k \int_{(\R^d)^k} m(x_i)^k \ \ind\{G_S(\xb) \text{ \rm is connected}\} d\xb.
\end{align}
Now, if $G_S(\xb)$ is connected for some $\xb=(x_1,\ldots,x_k)\in(\R^d)^k$, then it follows that $\lVert x_i-x_j\rVert \leq (k-1)\theta\rho$ for all $i,j\in\{1,\ldots,k\}$. Hence, the last expression in the above display is upper bounded by
\begin{align*}
& \sum_{i=1}^k \int_{(\R^d)^k} m(x_i)^k \ \ind\{x_j \in B(x_i, (k-1) \theta\rho) \text{ for all } j\} d\xb \\ &=  \sum_{i=1}^k \int_{\R^d} m(x_i)^k \ \lambda\left(B(0, (k-1) \theta\rho)\right)^{k-1}dx_i\\ &= k\lambda\left(B(0, (k-1) \theta\rho)\right)^{k-1} \int_{\R^d} m(x)^k dx < \infty.
\end{align*}
\qed\end{po_prop_integrableSuff}

\subsection{Proofs for the asymptotic behavior of the expectation}

\newproof{po_prop_UFsparse}{Proof of Proposition \ref{prop:UFsparse}}
\begin{po_prop_UFsparse}
Since the Lebesgue density $m$ of $\mu$ is assumed to be boun\-ded, we have $t\mu\left(\cup_{i=1}^k (\rho_t S+x_i)\right) \leq t \rho_t^d k \lambda(S) \lVert m \rVert _\infty$, where $\lambda$ denotes Lebesgue measure on $\R^d$. Hence, it follows from Lemma \ref{lem:prelim} together with \eqref{expU} that
\begin{align*}
\E U_t^A \geq k! \E F_t^A \geq e^{-t \rho_t^d k \lambda(S) \lVert m \rVert _\infty}\E U_t^A.
\end{align*}
Dividing this inequality by $\E U_t^A$ and taking the limit $t\to\infty$ yields the result since by assumption $\lim_{t\to\infty
}t\rho_t^d \to 0$.
\qed\end{po_prop_UFsparse}

Apart from minor modifications, the two upcoming proofs are very similar to the proofs of \cite[Proposition 3.1 and Proposition 3.3]{P_2003} as well as \cite[Theorem 4.2 (i)]{BR_2015}. We will therefore present these proofs very briefly just for the sake of completeness.

\newproof{po_thm_expsparse}{Proof of Theorem \ref{thm:expsparse}}
\begin{po_thm_expsparse}
Integrability of all $U_t^A$ follows from Proposition \ref{prop:integrable} and Pro\-position \ref{prop:integrableSuff} together with the assumption $\gamma k > d$. The expectation of $U_t^A$ is given by the integral in \eqref{expU}. The change of variables
\begin{align*}
\Phi_t:(x_1,\ldots,x_k)\mapsto (x_1,x_1+\rho_t x_2, \ldots, x_1+\rho_t x_k)
\end{align*}
together with translation invariance of the set $A$ yields
\begin{align*}
\E U_t^A = t^k \rho_t^{d(k-1)} \int_{(\R^d)^k} I(\xb)m^{\otimes k}(\Phi_t(\xb)) d\xb,
\end{align*}
where
\begin{align*}
I(\xb) &= \ind\{\{0,x_2, \ldots, x_k\}\in A, G_ {S}(0,x_2,\ldots,x_k) \text{ \rm is connected}\}.\\
\end{align*}
Since $m$ is almost everywhere continuous, for almost every $\xb\in(\R^d)^k$, we have
\begin{align*}
\lim_{t\to\infty} m^{\otimes k}(\Phi_t(\xb)) = m(x_1)^k.
\end{align*}
Moreover, observe that for any $x,y\in\R^d$,
\begin{align*}
(\lVert y\rVert + 1)^{-\gamma} \leq (1 + \lVert x - y \rVert)^\gamma (\lVert x\rVert + 1)^{-\gamma}.
\end{align*}
Using this together with the assumption $m(x)\leq \alpha(\lVert x \rVert +1)^{-\gamma}$ and the fact that $I(\xb)\neq 0$ implies $\lVert x_i \rVert \leq (k-1)\theta$ for $i=2,\ldots,k$, we obtain
\begin{align*}
&I(\xb) m^{\otimes k}(\Phi_t(\xb)) \leq (1 + R)^{(k-1)\gamma} \alpha^k I(\xb) (\lVert x_1\rVert + 1)^{-\gamma k},
\end{align*}
where $R=(k-1)\theta\sup_{t\in\N} \rho_t$. Since by assumption $\gamma k > d$, the right hand side in the above display is integrable. Hence, \eqref{eq:lim_uta} follows from the dominated convergence theorem. The statement in \eqref{eq:lim_fta} now follows from \eqref{eq:lim_uta} together with Proposition \ref{prop:UFsparse}.\qed\end{po_thm_expsparse}

\newproof{po_thm_expthermo}{Proof of Theorem \ref{thm:expthermo}}
\begin{po_thm_expthermo}
Integrability of the random variables $F_t^A$ follows from Pro\-position \ref{prop:integrableSuff} together with the assumption $\gamma k > d$. Using Lemma \ref{lem:prelim} together with the change of variables
\begin{align*}
\Phi_t: (x_1,\ldots,x_k) \mapsto (x_1,x_1+\rho_t x_2,\ldots,x_1+\rho_t x_k)
\end{align*}
and translation invariance of the set $A$, one obtains
\begin{align} \label{eq:intTrans2}
\E F_t^A = \frac{t^k\rho_t^{d(k-1)}}{k!} \int_{(\R^d)^k} I(\xb) e^{-t J_t(\xb)} m^{\otimes k}(\Phi_t(\xb)) d\xb,
\end{align}
where
\begin{align*}
I(\xb) &= \ind\{\{0,x_2, \ldots, x_k\}\in A, G_ {S}(0,x_2,\ldots,x_k) \text{ \rm is connected}\},\\
J_t(\xb) &= \mu\left(\left(\rho_tS + x_1\right) \cup \bigcup_{i=2}^k \left(\rho_tS + x_1 + \rho_t x_i\right) \right).
\end{align*}
To analyze the limit behavior of $J_t(\xb)$, let
\begin{align*}
m_t^{\rm max} = \sup_{y\in B(x_1, k\theta \rho_t)} m(y) \ \ \text{ and } \ \ m_t^{\rm min} = \inf_{y\in B(x_1, k\theta \rho_t)} m(y).
\end{align*}
Then, for $\xb\in(\R^d)^k$ satisfying $I(\xb)\neq 0$, we have
\begin{align*}
m_t^{\rm min} t  V_t \leq tJ_t(\xb) \leq m_t^{\rm max} t V_t,
\end{align*}
where
\begin{align*}
V_t = \lambda\left(\left(\rho_tS + x_1\right) \cup \bigcup_{i=2}^k \left(\rho_tS + x_1 + \rho_t x_i\right)\right).
\end{align*}
By translation invariance and homogeneity of the Lebesgue measure $\lambda$, one has
\begin{align*}
V_t = \rho_t^d \lambda\left(S\cup \bigcup_{i=2}^k (S+x_i)\right).
\end{align*}
Taking into account that, since $m$ is almost everywhere continuous, both $m_t^{\rm min}$ and $m_t^{\rm max}$ converge to $m(x_1)$ as $t\to\infty$ for a.e. $x_1\in\R^d$, and that moreover by assumption $\lim_{t\to\infty} t\rho_t^d = c \in(0,\infty)$, it follows from the last three displays that for a.e. element $\xb\in(\R^d)^k$ satisfying $I(\xb)\neq 0$, we have
\begin{align*}
\lim_{t\to\infty} t J_t(\xb) = c \lambda\left(S\cup \bigcup_{i=2}^k (S+x_i)\right) m(x_1).
\end{align*}
Using again that $m$ is almost everywhere continuous, we derive from this that for a.e. $\xb\in(\R^d)^k$, the limit of the integrand in \eqref{eq:intTrans2} is given by
\begin{align*}
\lim_{t\to\infty} I(\xb) e^{-t J_t(\xb)} m^{\otimes k}(\Phi_t(\xb)) = I(\xb) m(x_1)^k e^{-c \lambda(S\cup(S+x_2)\cup\ldots\cup(S+x_k)) m(x_1)}.
\end{align*}
Moreover, since $I(\xb) e^{-t J_t(\xb)} m^{\otimes k}(\Phi_t(\xb))\leq I(\xb)m^{\otimes k}(\Phi_t(\xb))$, we can continue in the same way as in the proof of Theorem \ref{thm:expsparse} to deduce that the dominated convergence theorem applies. The result follows.
\qed\end{po_thm_expthermo}

\newproof{po_thm_expdense}{Proof of Theorem \ref{thm:expdense}}
\begin{po_thm_expdense}
To shorten the notation a bit, we will only consider the case $\alpha=1$. The general case is obtained entirely analogously.

First note that integrability of all $F_t^A$ is immediate from Proposition \ref{prop:integrableSuff} together with the assumption $\gamma k> d$. It follows from Lemma \ref{lem:prelim} together with the change of variables
\begin{align*}
\Phi_t: (x_1,\ldots,x_k) \mapsto ((t\rho_t^ d)^{1/\gamma}x_1, (t\rho_t^d)^{1/\gamma}x_1 + \rho_t x_2, \ldots, (t\rho_t^ d)^{1/\gamma}x_1 + \rho_t x_k)
\end{align*}
that
\begin{align} \label{eq:intTrans}
\nonumber k! \E F_t^A &= \int_{(\R^d)^k} \ind\{\xb\in \rho_tA, G_{\rho_tS}(\xb) \text{ \rm is connected}\} e^{-t\mu\left(\cup_{i=1}^k(\rho_tS+x_i)\right)} d(t\mu)^k(\xb)\\
&= \frac{t^k (t\rho_t^d)^{d/\gamma} \rho_t^{d(k-1)}}{(t\rho_t^d)^k}\int_{(\R^d)^k} (t\rho_t^d)^k I(\xb) e^{-tJ_t(\xb)} m^{\otimes k}(\Phi_t(\xb))d\xb,
\end{align}
where
\begin{align*}
I(\xb) &= \ind\{\Phi_t(\xb)\in\rho_t A, G_{\rho_t S}(\Phi_t(\xb))\text{ \rm is connected}\},\\
J_t(\xb) &= \mu\left(\left(\rho_tS+(t\rho_t^d)^{1/\gamma} x_1\right) \cup \bigcup_{i=2}^k \left(\rho_tS+(t\rho_t^d)^{1/\gamma} x_1 + \rho_t x_i\right) \right).
\end{align*}
The reason why $I$ is not indexed by $t$ is that $I$ is actually independent of $t$. Indeed, by translation invariance of $A$, one has
\begin{align*}
&\ind\{\Phi_t(\xb)\in\rho_t A, G_{\rho_t S}(\Phi_t(\xb))\text{ \rm is connected}\}\\ &= \ind\{\{0,x_2,\ldots,x_k\}\in A, G_{S}(0,x_2,\ldots,x_k)\text{ \rm is connected}\}.
\end{align*}
Moreover,
\begin{align} \label{eq:bounddens}
(t\rho_t^d)^k m^{\otimes k}(\Phi_t(\xb)) = (\lVert x_1\rVert + (t\rho_t^d)^{-1/\gamma})^{-\gamma} \prod_{i=2}^k (\lVert x_1 + \rho_t(t\rho_t^d)^{-1/\gamma} x_i \rVert + (t\rho_t^d)^{-1/\gamma})^{-\gamma}.
\end{align}
From this together with $\lim_{t\to\infty}t\rho_t^d = \infty$, we obtain
\begin{align} \label{eq:limdens}
\lim_{t\to\infty}  (t\rho_t^d)^k m^{\otimes k}(\Phi_t(\xb)) = \lVert x_1\rVert^{-\gamma k}.
\end{align}
We continue with the investigation of the limit behavior of $I(\xb)e^{-tJ_t(\xb)}$. Observe that for any $x,y\in\R^d$, one has
\begin{align*}
m(y) \leq (1 + \lVert x - y \rVert)^\gamma m(x).
\end{align*}
Moreover, the assumption $S\subseteq B(0,\theta)$ ensures that $I(\xb) = 1$ implies $\lVert x_i\rVert \leq (k-1) \theta$ for all $i=2,\ldots,k$. It follows that, if $I(\xb)=1$, then
\begin{align*}
J_t(\xb) &= \int_{\R^d}\ind\left\{y\in \left(\rho_tS+(t\rho_t^d)^{1/\gamma} x_1\right) \cup \bigcup_{i=2}^k \left(\rho_tS+(t\rho_t^d)^{1/\gamma} x_1 + \rho_t x_i\right)\right\} m(y) dy\\
&\leq \lambda \left(\left(\rho_tS+(t\rho_t^d)^{1/\gamma} x_1\right) \cup \bigcup_{i=2}^k \left(\rho_tS+(t\rho_t^d)^{1/\gamma} x_1 + \rho_t x_i\right)\right) c_t m\left((t\rho_t^d)^{1/\gamma} x_1\right),
\end{align*}
where $\lambda$ denotes Lebesgue measure on $\R^d$ and
\begin{align*}
c_t = (1+k\theta\rho_t)^\gamma.
\end{align*}
Now, translation invariance and homogeneity of $\lambda$ give
\begin{align*}
\lambda \left(\left(\rho_tS+(t\rho_t^d)^{1/\gamma} x_1\right) \cup \bigcup_{i=2}^k \left(\rho_tS+(t\rho_t^d)^{1/\gamma} x_1 + \rho_t x_i\right)\right) = \rho_t^d \lambda \left(S \cup \bigcup_{i=2}^k \left(S + x_i\right)\right).
\end{align*}
Combining the last three displays yields
\begin{align*}
t J_t(\xb) \leq (1+k\theta\rho_t)^\gamma \left(\lVert x_1 \rVert + (t\rho_t^d)^{-1/\gamma}\right)^{-\gamma} \lambda \left(S \cup \bigcup_{i=2}^k \left(S + x_i\right)\right).
\end{align*}
Similarly, one also obtains
\begin{align} \label{eq:tJlower}
t J_t(\xb) \geq (1+k\theta\rho_t)^{-\gamma} \left(\lVert x_1 \rVert + (t\rho_t^d)^{-1/\gamma}\right)^{-\gamma} \lambda \left(S \cup \bigcup_{i=2}^k \left(S + x_i\right)\right).
\end{align}
Taking into account that $\lim_{t\to\infty} (1+k\theta\rho_t)^{\gamma} = \lim_{t\to\infty} (1+k\theta\rho_t)^{-\gamma} = 1$ and that $\lim_{t\to\infty} t\rho_t^d = \infty$, we derive from the last two displays together with \eqref{eq:limdens} that the integrand of the integral in \eqref{eq:intTrans} converges,
\begin{align} \label{eq:limintegrand}
\lim_{t\to\infty} (t\rho_t^d)^k I(\xb) e^{-tJ_t(\xb)} m^{\otimes k}(\Phi_t(\xb)) = I(\xb) \lVert x_1\rVert^{-\gamma k} e^{-\lVert x_1\rVert^{-\gamma} \lambda \left(S \cup \bigcup_{i=2}^k \left(S + x_i\right)\right)}.
\end{align}
We aim to use the dominated convergence theorem. For this, we first conclude that if $I(\xb)\neq 0$ for some $\xb=(x_1,\ldots,x_k)\in(\R^d)^k$, then for $i=2,\ldots,k$, one has
\begin{align*}
&\frac{(\lVert x_1 + \rho_t(t\rho_t^d)^{-1/\gamma} x_i\rVert + (t\rho_t^d)^{-1/\gamma})^{-\gamma}}{(\lVert x_1\rVert + (t\rho_t^d)^{-1/\gamma})^{-\gamma}} = \left(\frac{\lVert (t\rho_t^d)^{1/\gamma} x_1\rVert + 1}{\lVert (t\rho_t^d)^{1/\gamma} x_1 + \rho_t x_i\rVert + 1}\right)^{\gamma}\\ &\leq \left(\frac{\lVert (t\rho_t^d)^{1/\gamma} x_1 + \rho_t x_i\rVert + \rho_t (k-1)\theta + 1}{\lVert (t\rho_t^d)^{1/\gamma} x_1 + \rho_t x_i\rVert + 1}\right)^{\gamma} \leq (1 + \rho_t (k-1)\theta)^\gamma.
\end{align*}
Observe that it follows from the above estimate together with \eqref{eq:tJlower} and \eqref{eq:bounddens} that there are constants $C, C'>0$, such that for all $t\in\N$,
\begin{align*}
(t\rho_t^d)^k I(\xb) e^{-tJ_t(\xb)} m^{\otimes k}(\Phi_t(\xb)) \leq C I(\xb) (\lVert x_1\rVert + (t\rho_t^d)^{-1/\gamma})^{-\gamma k} e^{-C' (\lVert x_1\rVert + (t\rho_t^d)^{-1/\gamma})^{-\gamma}}.
\end{align*}
Now, since the map $z\mapsto z^k e^{-C' z}$ is bounded above by some constant $C''>0$, we obtain that for any $t\in\N$,
\begin{align*}
(t\rho_t^d)^k I(\xb) e^{-tJ_t(\xb)} m^{\otimes k}(\Phi_t(\xb)) \leq C I(\xb) \min\left(\lVert x_1\rVert^{-\gamma k}, C''\right).
\end{align*}
By assumption we have $\gamma k>d$. Also, for any $\xb=(x_1,\ldots,x_k)\in(\R^d)^k$, one has $I(\xb) = 0$ if $\lVert x_i\rVert > (k-1)\theta$ for some $i\in\{2,\ldots,k\}$. This implies that the right hand side in the above display is integrable with respect to $\lambda^k$. The result follows from the dominated convergence theorem together with \eqref{eq:intTrans} and \eqref{eq:limintegrand}.
\qed\end{po_thm_expdense}

\subsection{Proofs for the strong laws of large numbers}
We will use the following well known result which is a consequence of the Borel-Cantelli lemma (see e.g. \cite[Theorem 3.18]{Ka_2002}).

\begin{lem} \label{BCLem}
Consider a sequence $(X_n)_{n\in\N}$ of real random variables and let $(a_n)_{n\in\N}$ be a sequence of real numbers such that $\lim_{n\to\infty} a_n = a$ for some $a\in\R$. Assume that
\begin{align*}
\sum_{n=1}^\infty \P(|X_n - a_n|\geq\epsilon) < \infty \ \text{ for any } \epsilon>0.
\end{align*}
Then
\begin{align*}
X_n \overset{a.s.}{\longrightarrow} a \ \ \text{as} \ \ n\to\infty.
\end{align*}
\end{lem}

\newproof{po_thm_stronglaws}{Proof of Theorem \ref{thm:stronglaws}}
\begin{po_thm_stronglaws}

[Proof of (i)] It follows from the concentration estimates in Corollary \ref{cor:conc} that for $\epsilon>0$,
\begin{align*}
\P\left(\left|\frac{F_t^A}{t^k \rho_t^{d(k-1)}} - \frac{\E F_t^A}{t^k \rho_t^{d(k-1)}}\right| \geq \epsilon\right) \leq p_u(t,\epsilon) + p_l(t,\epsilon),
\end{align*}
where
\begin{align*}
p_u(t,\epsilon) &= \exp\left(-\frac{t^k\rho_t^{d(k-1)}\epsilon^2}{a_t\left(2\frac{\E F_t^A}{t^k\rho_t^{d(k-1)}} +  \epsilon\right)}\right),\\
p_l(t,\epsilon) &= \exp\left(-\frac{(t^k\rho_t^{d(k-1)})^2\epsilon^2}{2 \max(a_t,4c_{\rho_t S} / 3) \E F_t^A}\right)
\end{align*}
and
\begin{align*}
a_t = k\left(c_{\rho_tS}^2 \sigma_{\rho_tS}^{t\mu}+ 1\right).
\end{align*}
Observe that $c_{\rho_tS} = c_S$. Moreover,
\begin{align*}
\sigma_{\rho_t S}^{t\mu} = \sup_{x\in\R^d} t\mu(\rho_t S + x) \leq t\rho_t^d \lambda(S) \lVert m\rVert_\infty
\end{align*}
and since by assumption $t\rho_t^d \to {\rm const} \in[0,\infty)$ as $t\to\infty$, it follows that the sequence $a_t$ is bounded. According to Theorem \ref{thm:expsparse} and Theorem \ref{thm:expthermo}, the sequence $\E F_t^A/(t^k\rho_t^{d(k-1)})$ converges to some $b\in[0,\infty)$. We conclude that there are constants $C>0$ and $C'>0$ such that for all $t\in\N$,
\begin{align*}
p_u(t,\epsilon) + p_l(t,\epsilon) \leq  \exp\left(-t^k\rho_t^{d(k-1)} C\right) + \exp\left(-t^k\rho_t^{d(k-1)} C'\right).
\end{align*}
Now, condition \eqref{condRegimeST} ensures that
\begin{align*}
\sum_{t\in\N} (p_u(t,\epsilon) + p_l(t,\epsilon)) <\infty.
\end{align*}
Invoking Lemma \ref{BCLem} and noting again that $\lim_{t\to\infty} \E F_t^A/(t^k\rho_t^{d(k-1)}) = b$ yields the result.
\smallskip

[Proof of (ii)]
Using again the concentration estimates of Corollary \ref{cor:conc} similarly as in the proof of (i), we obtain that for $\epsilon>0$,
\begin{align*}
\P\left(\left|\frac{F_t^A}{t(t\rho_t^d)^{d/\gamma-1}} - \frac{\E F_t^A}{t(t\rho_t^d)^{d/\gamma-1}}\right| \geq \epsilon\right) \leq p_u(t,\epsilon) + p_l(t,\epsilon),
\end{align*}
where
\begin{align*}
p_u(t,\epsilon) &= \exp\left(-\frac{t(t\rho_t^d)^{d/\gamma-1}\epsilon^2}{a_t\left(2\frac{\E F_t^A}{t(t\rho_t^d)^{d/\gamma-1}} +  \epsilon\right)}\right),\\
p_l(t,\epsilon) &= \exp\left(-\frac{(t(t\rho_t^d)^{d/\gamma-1})^2\epsilon^2}{2 \max(a_t,4c_{S} / 3) \E F_t^A}\right)
\end{align*}
and
\begin{align*}
a_t = k\left(t\rho_t^d c_{S}^2 \lambda(S) \lVert m\rVert_\infty+ 1\right).
\end{align*}
By Theorem \ref{thm:expdense}, the sequence $\E F_t^A / (t(t\rho_t^d)^{d/\gamma - 1})$ converges to $\mathfrak{d}_S^A(m)<\infty$. Using this together with the assumption $\lim_{t\to\infty}t\rho_t^d = \infty$, we conclude that there are constants $C,C'>0$ such that for sufficiently large $t$,
\begin{align*}
p_u(t,\epsilon) + p_l(t,\epsilon) \leq  \exp\left(-t(t\rho_t^d)^{d/\gamma-2} C\right) + \exp\left(-t(t\rho_t^d)^{d/\gamma-2} C'\right).
\end{align*}
Hence, the assumption \eqref{condRegimeD} ensures
\begin{align*}
\sum_{t\in\N} (p_u(t,\epsilon) + p_l(t,\epsilon)) <\infty.
\end{align*}
Recall that $\lim_{t\to\infty} \E F_t^A / (t(t\rho_t^d)^{d/\gamma - 1}) = \mathfrak{d}_S^A(m)$ and apply Lemma \ref{BCLem} to obtain the result.
\qed\end{po_thm_stronglaws}

\section{Acknowledgments}

The author thanks Günter Last for suggesting the topic of the present paper. Moreover, the author thanks Giovanni Peccati and Matthias Reitzner for many useful remarks that helped improving the presentation of the results. The author is partially supported by the German Research Foundation DFG-GRK 1916.

\bibliography{references}

\end{document}